\documentclass[12pt]{amsart}
\usepackage[utf8]{inputenc}
\usepackage[margin=1.5in]{geometry} 
\usepackage{amsmath, amsthm, amssymb, bbm, bm, amsfonts, enumerate, xcolor, hyperref}
\usepackage[colorinlistoftodos]{todonotes}
\usepackage{caption}
\usepackage{mathrsfs}

\hypersetup{colorlinks,
	linkcolor=green!50!black,
	citecolor=blue!70!black,
	urlcolor=red!50!black
}

\def \Fq {{\mathbb{F}_{q}}}

\def \N {{\mathbb N}}
\def \C {{\mathbb C}}
\def \Z {{\mathbb Z}}
\def \Q {{\mathbb Q}}

\def \p {{\mathfrak p}}
\makeatother
\newcommand{\fq}{\mathbb{F}_q}
\newtheorem{theorem}{Theorem}[section]
\newtheorem{definition}[theorem]{Definition}
\newtheorem{lemma}[theorem]{Lemma}

\newtheorem{proposition}[theorem]{Proposition}

\newtheorem{cor}[theorem]{Corollary}
\numberwithin{equation}{section}
	\newcounter{countknownthm}

\title[Primitive Root Conjecture]{On Artin's Primitive Root Conjecture for Function Fields over $\fq$}



\author[L. Hochfilzer]{Leonhard Hochfilzer}
\address{L. Hochfilzer,
Mathematisches Institut,
Bunsenstraße 3-5, 37073 Göttingen}
\email{leonhard.hochfilzer@uni-goettingen.de}

\author[E. Waxman]{Ezra Waxman}
\address{E. Waxman,
	University of Haifa, Department of Mathematics, 199 Aba Khoushy Ave., Mt. Carmel, Haifa, 3498838}
\email{ezrawaxman@gmail.com}

\date{June 7, 2023}

\begin{document}

\begin{abstract}
In 1927, E. Artin proposed a conjecture for the natural density of primes $p$ for which $g$ generates $(\mathbb{Z}/p\mathbb{Z})^\times$.  By carefully observing numerical deviations from Artin’s originally predicted asymptotic, Derrick and Emma Lehmer (1957) identified the need for an additional correction factor; leading to a modified conjecture which was eventually proved to be correct by Hooley (1967) under the assumption of the generalised Riemann hypothesis.  An appropriate analogue of Artin's primitive root conjecture may moreover be formulated for an algebraic function field $K$ of $r$ variables over $\fq$.  Relying on a soon to be established theorem of Weil (1948), Bilharz (1937) provided a proof in the particular case that $K$ is a global function field (i.e. $r=1$), which is correct under the assumption that $g \in K$ is a \textit{geometric} element.  Under these same assumptions, Pappalardi and Shparlinski (1995) established a quantitative version of Bilharz's result.  In this paper we build upon these works by both generalizing to function fields in $r$ variables over $\fq$ and removing the assumption that $g \in K$ is geometric; thereby completing a proof of Artin's primitive root conjecture for function fields over $\fq$.  In doing so, we moreover identify an interesting correction factor which emerges when $g$ is not geometric.  A crucial feature of our work is an exponential sum estimate over varieties that we derive from Weil's Theorem.
\end{abstract}

\maketitle

\section{Introduction}
\subsection{Classical Primitive Root Conjecture}
Consider an odd prime $p \in \mathbb{N}$, and recall that the multiplicative group $(\Z/p\Z)^{\times}$ is cyclic of order $p-1$.  We say that $g \in \Q \setminus \{0\}$ is a \textit{primitive root} mod $p$ (denoted $\mathrm{ord}_{p}(g) = p-1$) if the $p-$adic valuation satisfies $v_{p}(g) = 0$, and if the reduction $g$ mod $p$ generates the group $(\Z/p\Z)^{\times}$.  If $g \in \Q^{\times,2}$ is a perfect square or $g = -1$, then there are at most finitely many primes $p$ such that $\mathrm{ord}_{p}(g) = p-1$.  \textit{Artin's primitive root conjecture} states that in all other cases, $g$ is a primitive root mod $p$ for infinitely many primes $p$.

\subsection{Primitive Root Conjecture for  $\mathbb{F}_{q}(t)$}\label{Fq[T] Artin}
Artin first proposed his conjecture in a private conversation with Helmut Hasse on September 12, 1927. Hasse then assigned the problem to his PhD student, Herbert Bilharz, who began working on the problem in 1933.  Shortly afterwards, Erd\H{o}s announced a proof of the conjecture, conditional on the generalized Riemann hypothesis for certain Dedekind zeta functions.  Erd\H{o}s did not, in fact, publish any formal paper on the topic; but the threat was nonetheless sufficient motivation for Bilharz to shift the topic of his dissertation to the analogous problem over global function fields.  For more on this interesting history see \cite[Appendix 2]{cojocaro2003cyclicity}.
\\
\\
The simplest instance of Artin's conjecture in the function field setting is as follows.  Let $\fq$ denote a finite field of $q$ elements and $\fq[t]$ the ring of polynomials with coefficients in $\fq$.  We moreover let $\mathcal{P}_{n} \subset \fq[t]$ denote the subset of prime monic polynomials of degree $n \in \mathbb{N}$.  For $g(t) \in \fq(t)$ and $P \in \mathcal{P}_{n}$ such that $v_{P}(g)=0$, we let $\textnormal{ord}_{P}(g(t))$ denote the order of $g(t)$ in the multiplicative group  $(\fq[t]/(P))^{\times}$, where $(P) \subseteq \fq[t]$ denotes the prime ideal generated by the prime $P$.  In particular, $g(t)$ generates $(\fq[t]/(P))^{\times}$ if and only if $\textnormal{ord}_{P}(g) = q^{n}-1$, in which case we say that $g$ is a \textit{primitive root} mod $P$.\\
\\
Two immediate obstructions prevent $g(t)$ from being a primitive root modulo infinitely many prime $P \in \fq[t]$.  First, note that if $g(t) \in \fq^{\times}$ is a constant in $\fq[t]$, then $\textnormal{ord}_{P}(g(t)) \leq q-1$, and therefore $g(t)$ cannot be a primitive root modulo $P \in \mathcal{P}_{n}$, whenever $n >1$.  Second, suppose $g(t) = h(t)^\ell$ for some prime $\ell \mid q-1$.  Since
\[q^{n}-1 = (q-1)(q^{n-1}+\dots +q+1),\]
we find that $\ell \mid q^n-1$ for any $n \in \N$.  Thus  $\mathrm{ord}_P(g) \leq \frac{q^n-1}{\ell} < q^n-1$, from which it follows that $g(t)$ cannot be a primitive root modulo any prime $P \in \fq[t]$.

In the $\fq[t]$ setting, \textit{Artin's primitive root conjecture} is then the following claim:\\
\\
\textbf{\textit{Primitive Root Conjecture for $\fq(t)$}}: \textit{Suppose $g(t) \in \fq(t) \setminus \fq$ is not an $\ell^{th}$ power, for any prime $\ell \mid q-1$. Then $g(t)$ is a primitive root mod $P$ for infinitely many prime $P \in \fq[t]$}.\\
\\
For certain particular choices of $g(t) \in \fq[t]$, the above conjecture may be proved using elementary methods.  For instance, when $g(t) = t^m + c$, Jensen and Murty \cite{JensenMurty2000} offer a proof that relies only on the theory of Gauss sums.  When $g(t) \in \fq[t]$ is irreducible, Kim and Murty \cite{KimMurty2020, KimMurty2020II} provide a proof that relies only on establishing a zero-free region for relevant $L$-functions.

\subsection{Primitive Root Conjecture for Function Fields}
To formulate Artin's primitive root conjecture over more general function fields, it is appropriate to take a more geometric viewpoint (see Section~\ref{geometric background} for a 
rather self-contained overview of this geometric set-up.)
Let $X$ be a geometrically integral projective variety over $\Fq$ of dimension $r >0$, and write $K = \Fq(X)$ for its function field. Given a point $\mathfrak{p} \in X$, denote by $\mathcal{O}_{\mathfrak{p}} = \mathcal{O}_{X,\mathfrak{p}}$ the stalk of the structure sheaf $\mathcal{O}_{X}$ at $\mathfrak{p}$. Abusing notation, we moreover write $\mathfrak{p} \subset \mathcal{O}_{\mathfrak{p}}$ to denote the unique maximal ideal, and then let $\kappa_{\mathfrak{p}} = \mathcal{O}_{\mathfrak{p}}/\mathfrak{p}$ denote the corresponding residue field. If $\p \in X$ is moreover closed, then $\kappa_{\mathfrak{p}}$ is a finite field extension of the base field of $X$, i.e. $\deg \mathfrak{p} := [\kappa_{\mathfrak{p}} \colon \Fq]$ is finite.\\
\\
Fix $g \in K$, and let $\mathfrak{p}$ denote a closed point of $X$. We say $g$ is \textit{regular} at $\mathfrak{p}$ if $g$ lies in the image of the embedding $\mathcal{O}_{\mathfrak{p}} \hookrightarrow K$. Upon pulling $g$ back under this embedding, we may then consider $g \in \kappa_{\mathfrak{p}}$. We say that $g \in K$ is a \emph{primitive root modulo $\mathfrak{p}$} if $g$ is regular at $\mathfrak{p}$ and if $g$ generates the multiplicative group $\kappa_{\mathfrak{p}}^{\times}$.  In such a case, we moreover refer to $\p$ as an \textit{Artin prime} for $g$.\\
\\
As with the case over $\fq[t]$, it may easily be shown that if $g$ lies in $\fq$ or is an $\ell^{th}$ for some prime $\ell \mid q-1$, then $g$ may be a primitive root modulo $\p$ for at most finitely many closed points $\p \in X$.  In this more general set-up, \textit{Artin's primitive root conjecture} is then the following claim:\\
\\
\textbf{\textit{Primitive Root Conjecture for $K$}}: \textit{Let 
$K = \fq(X)$ denote the function field of some geometrically integral projective variety $X$ over $\fq$.  Suppose moreover that $g \in K \setminus \fq$ is not an $\ell^{th}$ power, for any prime $\ell \mid q-1$. Then for any such $X$, $g$ is a primitive root modulo $\p$, for infinitely many closed points $\p \in X$.}\\
\\
Bilharz ~\cite{bilharz1937primdivisoren} addressed the special case in which $X$ is a geometrically integral projective curve over $\fq$ (i.e. the case in which $K$ is a global function field).  His result was conditional on the Riemann hypothesis for global function fields $-$ a theorem subsequently established by André Weil \cite{weil1948courbes}.  Bilharz's proof contains a gap, however, for cases in which $g \in K$ is not a \textit{geometric element} (see Definition \ref{geometric element}).  Though this oversight has been previously noted, no corrected proof for the case in which $g \in K$ is geometric has thus far appeared in the literature (see ~\cite[pp. 155]{rosen2002number} for a more detailed discussion).\\
\\
In our work, we remove the assumption that $g \in K$ be a geometric element.  Applying an exponential sum estimate derived from Weil's theorem (Proposition \ref{prop:exponential_sum_estimate}) we moreover generalize to projective varieties of arbitrary dimension.  Note that every function field over $\fq$ (i.e. every field extension $K/\fq$ of positive finite transcendence degree) may be recovered as $K = \fq(X)$, where $X$ is some geometrically integral projective variety of dimension $r > 0$ over $\fq$.  We thus in fact complete a proof of the following theorem:
\begin{theorem}
[\textit{\textbf{Primitive Root Conjecture for Function Fields over $\fq$}}]\label{Artin Conjecture}
Let $K$ denote any function field over $\fq$.  Then Artin's primitive root conjecture holds for $K$.
\end{theorem}

Several additional variations of Artin's primitive root conjecture have been studied over function fields; for example, over Carlitz modules~\cite{Eisenstein2020, hsu_1997_carlitz}, rank one Drinfeld modules~\cite{HSU2001_drinfeld, KuoDrinfeld2022, YAO2010_drinfeld}, and one dimensional tori over function fields~\cite{chen2003primitive_tori}.
We further note that primitive polynomials are of interest for engineering applications involving LFSRs (linear-feedback shift registers).

\subsection{Quantitative Results}
As above, let $X$ denote a geometrically integral projective variety over $\fq$, and let $K$ denote its function field.  
If $K$ is a global function field and $g \in K$ geometric, Bilharz demonstrated that the set of Artin primes for $g \in K$ has positive Dirichlet density, i.e. that
\[\lim_{s \rightarrow 1}\frac{\sum_{\p \in P_{g}}|\kappa_{\p}|^{-s}}{\sum_{\p \in X}|\kappa_{\p}|^{-s}}>0,\]
where the sum in the denominator runs over the closed points in $X$, and where $P_{g} \subseteq X$ denotes the set of Artin primes for $g \in K$.  From this, Bilharz then deduced Artin's primitive root conjecture for $K$.\\
\\
A more quantitative description for this count was provided by Pappalardi and Shparlinski.  Let $N_{X}(g,n)$ denote the number of closed points $\p \in X$ such that deg $\p = n$, and such that $g \in K$ is a primitive root modulo $\p$.  In the particular case that $X= \mathscr{C}$ is a non-singular projective curve and $g \in K$ is geometric, Pappalardi and Shparlinski provided an asymptotic description of $N_{\mathscr{C}}(g,n)$, proving that for any $\varepsilon > 0$,

\begin{equation}\label{Pap Shpar theorem}
N_{\mathscr{C}}(g,n) = \frac{\varphi(q^{n}-1)}{n}+O_{\varepsilon,g,\mathscr{C}} \left(q^{n/2(1+\varepsilon)}\right).
\end{equation}

In this work, we generalize the above result by providing an asymptotic count for $N_{X}(g,n)$ while removing the assumptions that $g \in \fq(X)$ is geometric; that $X$ is non-singular; and moroever allowing $X$ to a geometrically integral projective variety of arbitrary finite dimension $r \geq 1$.  Our main result is the following:

\begin{theorem}  \label{thm.APC_general}
	Let $X/\Fq$ be a geometrically integral projective variety of dimension $r \geq 1$
	with function field $K = \Fq(X)$. Let $g \in K$ be a rational function.  If $g \in \fq$ or	$g$ is a full $\ell^{th}$ power in $K$ for some rational prime $\ell \mid q-1$ then $N_{X}(g,n) = 0$ for all $n > 1$. Otherwise, for any $\varepsilon > 0$,
	\begin{equation} \label{eq.number_roots_degree_n}
	N_{X}(g,n) = \rho_g(n) \left( \frac{\varphi(q^n-1)q^{n(r-1)}}{n}  + O_{\varepsilon,X,g} \left(
q^{n(r-1/2+\varepsilon)}\right) \right),
	\end{equation}
 where $\rho_{g}(n)$ is as in \eqref{rho_constant}. 
\end{theorem}

By \cite[Theorem 2.9]{MontgomeryVaughan2007}, we note that $\varphi(q^n-1) \gg_\nu q^{n(1-\nu)}$ for any $\nu \in (0,1)$.  Thus \eqref{eq.number_roots_degree_n} indeed yields a main term, in the limit as $q^{n} \rightarrow \infty$.\\
\\
If $g \in K$ is a geometric element, then $\rho_g(n) = 1$ for all $n \geq 1$.  We thus recover equation (\ref{Pap Shpar theorem}) in the case that $X= \mathscr{C}$ is a projective curve (i.e. $r =1$).  For non-geometric $g \in K^{\times}$, it is possible that  $\rho_g(n) = 0$ for certain $n \in \mathbb{N}$. Nonetheless, we show that $\rho_g(n) \geq 1$ for infinitely many $n \in \mathbb{N}$, thereby confirming
Theorem \ref{Artin Conjecture} (see Section \ref{qualitative_results}).
\subsection{Comparison to Classical Setting}
Let $\mathcal{P} \subset \mathbb{N}$ denote the set of primes, and let $\mathcal{P}_{g}\subseteq \mathcal{P}$ denote the subset of \textit{Artin primes} for $g$, namely the primes $p \in \mathcal{P}$ for which $\mathrm{ord}_{p}(g) = p-1$.  When $g \in \N$ is not an exact power, Artin conjectured that the natural density of $\mathcal{P}_{g} \subseteq \mathcal{P}$ is equal to
\[A:=\prod_{\substack{p \textnormal{ prime}}}\left(1-\frac{1}{p(p-1)}\right)\approx .3739558,\]
a value known as the \textit{Artin constant}.  Due to careful numerical observations pioneered by Derrick and Emma Lehmer, it later emerged that, for certain $g$, an additional correction factor is needed.  Slightly more generally, the natural density of $\mathcal{P}_{g}\subseteq \mathcal{P}$ is conjectured to equal $c_{g}A_{h}$, where $A_{h}$ is an explicit Euler product, and $c_{g} \in \Q$.  $A_{h}$ is moreover a \textit{linear} factor, which depends on the value of the largest integer $h \geq 1$ such that $g$ is an $h^{th}$ power in $\Z$, while $c_{g}$ is a \textit{quadratic} correction factor that takes into account \textit{entanglements} between number fields of the form $\Q(\zeta_{\ell},g^{1/\ell})$, $\ell \in \mathcal{P}$.  This modified conjecture was eventually proven correct by Hooley~\cite{hooley_artin} under the assumption of the generalised Riemann Hypothesis (GRH) for Dedekind $\zeta$-functions.\\
\\
Returning to the function field setting, let $P_{n}$ denote the closed points of $X$ of fixed degree $n$, and let $P_{n}(g) \subseteq P_{n}$ denote the subset of Artin primes for $g \in K$; so that $\#P_{n}(g)=N_{X}(g,n)$. When $X$ is a (non-singular) curve and $g$ is geometric, it follows from \eqref{Pap Shpar theorem} that the density of $P_{n}(g) \subseteq P_{n}$ is asymptotic to $A(n) := \varphi(q^{n}-1)q^{-n}$, in the limit as $q^n \rightarrow \infty$.  More generally, we find from \eqref{eq.number_roots_degree_n} that the density of $P_{n}(g) \subseteq P_{n}$ is asymptotic to $A_{g}(n) := \rho_g(n)\varphi(q^{n}-1)q^{-n}$, where $\rho_g(n)$ depends on the factorization behaviour of $g$ in $K \otimes_{\fq} \overline{\mathbb{F}}_{q}$.\\
\\
Note that $A(n)$ does not, in fact, converge in the limit as $n \rightarrow \infty$.  Indeed, even the natural density of Artin primes, namely the limit
\[\lim_{N \rightarrow \infty}\frac{\sum_{n=1}^{N}N_{X}(g,n)}{\sum_{n=1}^{N}\#P_{n}},\]
does not, in general, exist.  This was demonstrated by Bilharz~\cite{bilharz1937primdivisoren} and expanded upon by Perng~\cite{perng_natural_density}.
\subsection*{Structure of Paper}
The remainder of this paper is structured as follows. In Section \ref{geometric background} we provide an overview of the relevant geometric set-up in which to present Artin's primitive root conjecture for function fields over $\fq$.  In Section \ref{geom_extensions} we then discuss \textit{geometric extensions} and \textit{geometric elements} $g \in K$.
 In Section \ref{Exponential_Sum}, we use Weil's theorem to establish a very general estimate for exponential sums
over a variety; a crucial step necessary to extend our results to the general setting of projective varieties.  Section \ref{Proof_of_Theorem} then establishes a proof of Theorem~\ref{thm.APC_general}, and Section \ref{qualitative_results} provides a proof of Theorem~\ref{Artin Conjecture}. Finally, in Section \ref{Heuristic_Interpretation} we provide a heuristic interpretation of the counting function $N_{X}(g,n)$, in the case that $K=\fq(X)$ is a global function field; in order to draw a conceptual comparison between our correction factor, $\rho_{g}(n)$, and the classical correction factor, $c_{g}$.
\subsection*{Acknowledgments} We thank Jonas Baltes, Chris Hall, Seoyoung Kim, Pieter Moree, Michael Rosen, and Damaris Schindler for helpful discussions.  The second author was funded by a Zuckerman Post-doctoral fellowship at the University of Haifa.

\section{Background on Projective Schemes}\label{geometric background}
\subsection*{Projective Schemes}
A \textit{graded ring} is a ring $S$ endowed with a direct sum decomposition $S = \oplus_{d\geq 0}S_{d}$ of the underlying additive group, such that $S_{d}S_{e} \subset S_{d+e}$.  We say that a non-zero element $a \in S$ is \textit{homogeneous} of degree $d$, denoted $\textnormal{deg } a=d$, if $a \in S_{d}$.  A \textit{homogenous ideal} is an ideal $I \subset S$ that is generated by a set of homogenous elements.  The ideal consisting of elements of positive degree, namely $S_{+}:= \oplus_{d > 0}S_{d}$, is referred to as the \textit{irrelevant ideal}.  If $S=\oplus_{d\geq 0}S_{d}$ is a graded ring, and $I \triangleleft S$ a homogenous ideal, then the quotient ring $R = S/I$ is itself a graded ring, with $R_{d} = S_{d}/(I \cap S_{d})$.\\
\\
 Consider the set 
\[\textnormal{Proj}(S):= \{\p \subseteq S: \p \textnormal{ homogenous prime ideal}, \hspace{2mm} S_{+} \not \subset \p\}.\]
We define a topology on $X = \mathrm{Proj}(S)$ (called the \textit{Zariski topology}) by designating the closed sets of \textnormal{Proj}($S$) to be of the form
\[Z(I):=\{\p \in \textnormal{Proj}(S): I \subseteq \p\},\]
where $I \subset S$ denotes a homogenous ideal.  A point $\mathfrak{p} \in X$ is said to be a \emph{closed point} if $\overline{\{ \mathfrak{p} \}} = \{ \mathfrak{p} \}$, equivalently, if there is no $\mathfrak{q}\in X$ such that $ \p \subsetneq \mathfrak{q}$. \\
\\
The \textit{distinguished open set} associated to any homogenous element $f \in S_{+}$ is then given by 
\[X_{f}:=\textnormal{Proj}(S) \setminus Z(\langle f \rangle)= \{\p \in \textnormal{Proj}(S): f \not \in \p \},\]
and the collection of such sets, namely $\{X_{f}: f \in S_{+} \textnormal{ homogeneous}\}$, forms a basis for the topology on $\textnormal{Proj}(S)$.  The space \textnormal{Proj}$(S)$, together with its Zariski topology, is referred to as a \textit{projective scheme}.\\
\\ 
The \textit{structure sheaf}, denoted $\mathcal{O}_{X}$, is a \textit{sheaf} on Proj$(S)$, defined on the distinguished open sets as
\[
\mathcal{O}_X(X_{f}):= S_{(f)} = \left\{\frac{a}{f^{n}}: a \in S \textnormal{ is homogenous}, n \in \Z_{\geq 0}, \textnormal{ deg }a = n\cdot \textnormal{deg }f \right\},
\]
i.e. as the zero-degree component of the localization $\{1,f,f^{2},\dots\}^{-1}S$. The projective scheme $X = \mathrm{Proj}(S)$ is called \emph{integral} if $S_{(f)}$ is an integral domain for any homogeneous $f \in S_+$.  An integral projective scheme is, in particular, irreducible as a topological space, and we let dim$(X)$ denote the \textit{dimension} of $X$.
We then find that dim$(Y) < $ dim$(X)$, for any proper closed subset $Y \subset X$.  Finally, we note that any integral scheme $X$ has a \emph{generic point} $\eta$, that is, an element $\eta \in X$ such that $\overline{\{ \eta\}} = X$.

\subsection*{Function Fields, Stalks, and Residue Fields}
Let $X$ denote an integral projective scheme.  Then for any homogeneous $f,g \in S_+$, $\mathrm{Frac}(S_{(f)}) \cong \mathrm{Frac}(S_{(g)})$, where $\mathrm{Frac}(R)$ denotes the fraction field of an integral domain $R$. We can thus define $\mathcal{K}(X) := \mathrm{Frac}(S_{(f)})$, for any homogeneous $f \in S_+$, to be the \emph{function field} of $X$, which can moreover be expressed explicitly as
\[\mathcal{K}(X):= \left \{\frac{a}{b}: a, b \in S_{d} \textnormal{ for some } d \in \Z_{\geq 0}, b \neq 0 \right\}.\]
The \textit{stalk} at a point $\mathfrak{p} \in X$ refers to the \textit{local ring}
\[\mathcal{O}_{X,\mathfrak{p}} := \left \{\frac{a}{b} \in \mathcal{K}(X): a, b \in S_{d} \textnormal{ for some } d \in \Z_{\geq 0},b \not \in \mathfrak{p}\right\},\]
whose unique maximal ideal is given explicitly by
\[\p \mathcal{O}_{X,\mathfrak{p}}:=\left \{\frac{a}{b} \in \mathcal{K}(X): a, b \in S_{d} \textnormal{ for some } d \in \Z_{\geq 0}, a \in \p, b \not \in \p\right\}.\]
We refer to $\kappa_{\mathfrak{p}} := \mathcal{O}_{X,\mathfrak{p}}/\mathfrak{p}\mathcal{O}_{X,\mathfrak{p}}$ as the \textit{residue field} of $\p$.
The intersection of all stalks, namely 
\[\mathcal{O}_{X}(X):= \bigcap_{\p \in X} \mathcal{O}_{X,\mathfrak{p}},\]
is referred to as the \textit{global sections} of $X$.  We moreover say that $X$ is \textit{normal} if $\mathcal{O}_{X,\mathfrak{p}}$ is an integrally closed domain inside $\mathcal{K}(X)$, for every $\mathfrak{p} \in X$.\\
\\
Two noteworthy properties of the function field $\mathcal{K}(X)$ are as follows.  First, if $\eta \in X$ is the generic point of $X$, then $\mathcal{K}(X)$ is isomorphic to the stalk $\mathcal{O}_{X,\eta}$. Second, if Spec$(R) \subset X$ is an \textit{affine open}, then $R$ is an integral domain and $\mathcal{K}(X)\cong \mathrm{Frac}(R)$.

\subsection*{Projective Varieties}
A \textit{projective variety} $X$ over the field $k$ is a projective integral scheme of the form $X = \textnormal{Proj}(S)$, where $S = k[x_{0},\dots,x_{n}]/I$ is a finitely generated $k$-algebra, and where $I \subseteq k[x_{0},\dots, x_{n}]$ is a homogenous ideal. Under these assumptions, we note that $X$ is both \textit{noetherian} and \textit{separated}. We denote its function field by $k(X)$ and we note that $\textnormal{dim}(X)$ is equal to the transcendence degree of $k(X)$ over $k$.  A projective variety of dimension one is referred to as a \textit{projective curve} (over $k$). \\
\\
Let $S = k[x_{0},\dots,x_{n}]/I$ be as above.  $\textnormal{Proj}(S)$ is said to be \textit{geometrically integral} if $\mathrm{Proj}(\overline{S})$ is integral, where $\overline{S} = (k[x_{0},\dots, x_{n}]/I) \otimes_k \overline{k}$.  For example, if $f \in k[x_{0},x_{1},x_{2}]$ is homogeneous of positive degree and \textit{absolutely} irreducible (i.e. irreducible over $\overline{k}$), then Proj$(k[x_{0},x_{1},x_{2}]/\langle f \rangle)$ is a geometrically integral projective curve.  We moreover let $\overline{k}(X)$ denote the function field of $\textnormal{Proj}(\overline{S})$.\\
\\
Let $X=\mathrm{Proj}(S)$ be a geometrically integral projective variety over $\Fq$, and let $K = \Fq(X)$ denote the function field of $X$.  Fix $g \in K$ and let $\mathfrak{p} \in X$ be closed.  We say that $g$ is \textit{regular} at $\mathfrak{p}$ if $g \in \mathcal{O}_{X,\mathfrak{p}} \subset K$.  For $g \in \mathcal{O}_{X,\mathfrak{p}} \setminus \p \mathcal{O}_{X,\mathfrak{p}}$, let ord$_{\p}(g)$ denote the order of $g \textnormal{ mod } \p \mathcal{O}_{X,\mathfrak{p}}$ in the multiplicative group $\kappa_{\mathfrak{p}}^{\times}$.  Note that since $\p \in X$ is closed, $\kappa_{\p}$ is in fact a finite algebraic extension of $\fq$, and we define $\deg \mathfrak{p} := [\kappa_{\mathfrak{p}} \colon \Fq]$ to be the \textit{degree} of $\p \in X$.  We say that $g \in K$ is a \emph{primitive root modulo $\mathfrak{p}$} if $g$ is regular at $\mathfrak{p}$ and if $g \textnormal{ mod } \p \mathcal{O}_{X,\mathfrak{p}}$ generates the multiplicative group $\kappa_{\mathfrak{p}}^{\times}$.    
 Equivalently, we thus say that $g \in K$ is a primitive root modulo $\mathfrak{p}$ if $\textnormal{ord}_{\p}(g) = \#|\kappa_{\mathfrak{p}}^{\times}|= q^{\deg \mathfrak{p}}-1$.


\subsection*{Divisors and Valuations}
Let $X=\mathrm{Proj}(S)$ be a normal, geometrically integral, projective variety over $\Fq$.  In particular, $X$ is a noetherian normal integral separated scheme.  A \emph{prime divisor} $Y$ of $X$ is a closed integral subscheme $Y \subset X$ of codimension one, i.e. such that dim$(Y) = \textnormal{dim}(X)-1$. Let $Y$ be a prime divisor of $X$, and let $\eta_{Y}$ denote its generic point.  As noted in \cite[Lemma II.6.1]{hartshorne2013algebraic}, the stalk $\mathcal{O}_{X, \eta_Y}$ is a discrete valuation ring with $\mathrm{Frac}(\mathcal{O}_{X, \eta_Y}) = K$.  We thus find that for any prime divisor $Y \subset X$, the valuation $v_Y:\mathcal{O}_{X, \eta_Y}^{\times} \rightarrow \Z$ corresponding to $\mathcal{O}_{X, \eta_Y}$ may be extended to a function $v_Y:K^{\times} \rightarrow \Z$.  If $v_Y(g) = n \in \Z_{>0}$, we say that $g$ has a \textit{zero} along $Y$ (of order $n$), while if $v_Y(g) = n \in \Z_{<0}$, we say that $g$ has a \textit{pole} along $Y$ (of order $n$).  We moreover note that $v_Y(\mu)=0$ for all $\mu \in \fq^{\times}$.\\
\\
Given $g \in K^\times$, we find that $v_Y(g) = 0$ for all but finitely many prime divisors $Y \subset X$ \cite[Lemma II.6.1]{hartshorne2013algebraic}. We may therefore define the \textit{degree} of $g$ to be
\begin{equation}\label{eq:def_deg}
\deg(g) := \sum_{Y \subset X} \lvert v_{Y}(g) \rvert,
\end{equation}
where the sum runs over all prime divisors $Y$ of $X$.   Alternatively, $\deg(g)$ may be viewed as the number of poles and zeros of $g$ on $X$, counted with multiplicity.  Note that in the particular case in which $X$ is a curve, the set of prime divisors of $X$ is precisely given by the set of closed points in $X$.

\subsection*{Rational Points}
Let $R = \fq[x_{1},\dots,x_{m}]/I$, be a finitely generated $\fq$-algebra, where $I  \subseteq \fq[x_{1},\dots,x_{m}]$ is an ideal.  Denote by Spec$(R)$ the \textit{affine $\fq$-scheme}, which is the affine scheme whose underlying set is the collection of prime ideals in $R$. The closed points of $\text{Spec}(R)$ are given by the maximal ideals of $R$. For finitely generated $\fq$-algebras $R$ and $S$, we further recall that morphisms $\rho$: Spec$(S) \rightarrow \text{Spec}(R)$ between $\fq$-schemes are in one-to-one correspondence with the $\fq$-algebra homomorphisms $\rho^{\#}: R \rightarrow S$ induced by the $\Fq$-algebra structure.\\
\\
An $\mathbb{F}_{q^{n}}$-\textit{rational point} of Spec$(R)$ is an $\fq$-scheme morphism
\[\rho: \text{Spec}(\mathbb{F}_{q^{n}}) \rightarrow \text{Spec}(R),\]
which then corresponds to a homomorphism of $\fq$-algebras
\[\rho^{\#}: R = \fq[x_{1},\dots,x_{m}]/I \rightarrow \mathbb{F}_{q^{n}}.\]
In particular, $\rho^{\#}:R \rightarrow \mathbb{F}_{q^{n}}$ is a ring homomorphism with $\Fq \subseteq \mathrm{Im}(\rho^\#) \subseteq \mathbb{F}_{q^n}$.  By the first isomorphism theorem, we note that $\mathrm{Im}(\rho^\#) \cong R/\textnormal{ker}(\rho^{\#})$, where $\textnormal{ker}(\rho^{\#})\subset R$ is maximal, since $\mathrm{Im}(\rho^\#)$ is in fact a field.  Thus, to any $\mathbb{F}_{q^{n}}$-rational point $\rho$, we may associate a unique maximal ideal, $\mathfrak{m} := \textnormal{ker}(\rho^{\#})\subset R$.

 Conversely, suppose $\mathfrak{m}\subset R$ is a maximal ideal. Then $R/\mathfrak{m} \cong \mathbb{F}_{q^{m}}$, where $m = [\mathbb{F}_{q^{m}}:\mathbb{F}_{q}] = \textnormal{deg }\mathfrak{m}.$
  If $m \leq n$ one may then associate precisely $m$ different $\mathbb{F}_{q^{n}}$-rational points to the closed point $\mathfrak{m} \in \textnormal{Spec}(R)$ as follows. Note that there are precisely $m$ different $\Fq$-invariant inclusions  $\varphi \colon \mathbb{F}_{q^m} \hookrightarrow \mathbb{F}_{q^n}$ 
  coming from the elements in $\textnormal{Gal}(\mathbb{F}_{q^{m}}/\fq) \cong \Z/m\Z$.
  
 Let $\pi: R \rightarrow R/\mathfrak{m}$ denote the projection map.  Then for any $\varphi$ as above the $\fq$-algebra homomorphism $\rho^{\#}_{\varphi}: R \rightarrow \mathbb{F}_{q^{m}}$ given by $\rho^{\#}_{\varphi} = \varphi \circ \pi$ corresponds to a unique $\mathbb{F}_{q^{n}}$-rational point.  Thus each closed point $\mathfrak{m} \in \textnormal{Spec}(R)$ of degree $m$ gives rise to precisely $m$ distinct $\mathbb{F}_{q^{n}}$-rational points, $\rho^{\#}_{\varphi}$.\\
\\
Let $X$ be a geometrically integral projective variety over $\fq$ with function field $K = \fq(X)$.  We let $X(\mathbb{F}_{q^{n}})$ denote the set of $\mathbb{F}_{q^{n}}$-rational points of $X$, i.e. the set of $\fq$-scheme morphisms
\[
\rho \colon \mathrm{Spec} (\mathbb{F}_{q^{n}}) \rightarrow X,
\]
Each $\rho$ factors through some open affine subscheme Spec$(R) \subset X$, where $R$ is an integral domain and moreover $K \cong \mathrm{Frac}(R)$.  Upon restricting to the preimage of Spec$(R)$ under $\rho$, each $\rho$ induces an $\fq$-algebra homomorphism
\[
\rho^{\#} \colon R \rightarrow \mathbb{F}_{q^{n}}.
\]
We now describe how to evaluate an element $g \in K^{\times}$ at a rational point $\rho \in X(\mathbb{F}_{q^{n}})$.  Consider $g = a/b \in K^{\times}$, where $a,b \in R$ and $b \neq 0$. If $\rho^{\#}(b) \neq 0$, we evaluate
\begin{equation}\label{eq:def_rational_point_map}
g(\rho) := \frac{\rho^{\#}(a)}{\rho^{\#}(b)} \in \mathbb{F}_{q^{n}}.
\end{equation}
Otherwise, we say that $g$ has a \emph{pole} at $\rho$.

Recall that the closed point $\mathfrak{p}$ corresponding to $\rho$ is given by $\mathfrak{p} = \mathrm{ker}(\rho^{\#}) \subset R$. Clearly $\rho^{\#}(b) = 0$ precisely when $b \in \mathrm{ker}(\rho^{\#}) = \mathfrak{p}$. Hence $g = a/b$ is regular at $\mathfrak{p}$ if and only if $g$ does not have a pole at any of the corresponding rational points, $\rho$.\\
\\
As a concrete example, consider the case $X = \mathbb{P}_{\mathbb{F}_q}^1$. Then $K = \fq(t)$ and the closed points correspond to irreducible monic polynomials in $\fq[t]$ in addition to the point at infinity. The generic point $\eta$ is then given by the zero-ideal, $(0) \subset \fq[t]$.  A closed point $\p \in X$ of degree $n$ is the ideal generated by an irreducible polynomial $p(t) \in \fq[t]$ of degree $n$, which correspond to $n$ distinct $\mathbb{F}_{q^{n}}$-rational points, namely the roots of $p(t)$.  Given a fixed root $\rho$ of $p(t)$, the map $\rho^{\#}$ is then given by
\[\rho^{\#}: \fq[t] \twoheadrightarrow \fq[t]/(p(t))\xrightarrow{\sim} \mathbb{F}_{q^n}, \hspace{5mm} f \mapsto [f] \mapsto f(\rho).\]
Artin's primitive root conjecture for $K$ then reduces to the particular setting described in Section \ref{Fq[T] Artin}.

\subsection*{Number field analogue}
We conclude this section by noting that the classical version of Artin's primitive root conjecture (i.e. the case over number fields) may also be phrased in a geometric language. Let $K/\Q$ be a number field with ring of integers $\mathcal{O}_K$, and recall that the closed points of Spec$(\mathcal{O}_K)$ are the set of non-zero prime ideals of $\mathcal{O}_K$. The residue field of a closed point $P \in \mathrm{Spec}(\mathcal{O}_K)$, i.e. of a non-zero prime ideal of $\mathcal{O}_K$, is given by $\kappa_{P}=\mathcal{O}_K/P$.  Given $g \in K$ we say that $g$ is a primitive root modulo a non-zero prime ideal $P \in \mathrm{Spec} (\mathcal{O}_K)$ if $v_P(g) \geq 0$ and if $g$ generates the multiplicative group $\kappa_{P}^{\times}=(\mathcal{O}_K/P)^\times$. Artin's primitive root conjecture over $K$ states that for appropriate $g \in K$, the number of prime ideals $P$ for which $g$ is a primitive root is infinite.

\section{Geometric Extensions}\label{geom_extensions}
Let $X = \textnormal{Proj}(S)$ denote a geometrically integral projective variety over $\fq$, where char$(\fq) = p$. Let $K = \fq(X)$ denote its function field.  Recall that we write $\overline{\mathbb{F}}_q(X)$ to refer to the function field of the base change of $X$ to $\overline{\mathbb{F}}_q$, namely to the compositum of fields $K \overline{\mathbb{F}}_q$.  We moreover note that $\overline{\mathbb{F}}_q(X) \cong \mathbb{F}_q(X) \otimes_{\fq} \overline{\mathbb{F}}_q$.\\

Let $\overline{K}$ denote a fixed algebraic closure of $K$, and consider the algebraic field extensions $L/K$ and $M/\fq$.  Note that $L$ and $M$, as well as the compositum $LM$, may each be embedded inside $\overline{K}$. Given an algebraic field extension $L/K$ we thus write $\overline{\mathbb{F}}_q \cap L \subset \overline{K}$ to be the \textit{constant field} of $L$, namely the maximal algebraic subextension of $\fq$ inside $L$. In particular, since $K = \fq(X)$, where $X$ is integral of finite type over $\fq$, it follows from~\cite[Proposition 2.2.22]{poonen2017rational} that $K \cap \overline{\mathbb{F}}_q = \mathbb{F}_q$.\\
\begin{definition}
Let $K$ be as above and let $L_{2}/L_{1}/K$ be a tower of algebraic field extensions. We say that $L_{2}/L_{1}$ is a \textbf{geometric field extension} if $L_{2} \cap \overline{\mathbb{F}}_{q} = L_{1}\cap \overline{\mathbb{F}}_{q}$.  In particular, if $L/K$ is an algebraic field extension, we say that $L/K$ is a geometric field extension if  $L \cap \overline{\mathbb{F}}_{q} = \mathbb{F}_{q}$.
\end{definition}

\begin{definition}\label{geometric element}
Let $g\in K$.
We say that $g$ is \textbf{geometric at a rational prime} $\ell \neq  p$ if, for all roots $\alpha \in \overline{K}$ of the polynomial $X^\ell-g$, the extension $K(\alpha)/K$ is a proper geometric extension of fields. If
$g \in K$ is geometric at all primes $\ell\neq p$, we refer to $g \in K$ as \textbf{geometric}.
\end{definition}

\begin{lemma} \label{le.equivalent_characterisations_non_geometric}
	Let $K = \fq(X)$, let $g \in K$ and let $\ell \neq p$ be a rational prime. The following are equivalent:
	\begin{enumerate}[(i)]
		\item $g$ is not geometric at $\ell$.
		\item There exists $\mu \in \fq$ and $b \in K$ such that $g = \mu b^\ell$.
		\item There exists $\tilde{g} \in K \overline{\mathbb{F}}_q$ such that $g = \tilde{g}^\ell$.
	\end{enumerate}
\end{lemma}

\begin{proof} 

	(i) $\implies$ (ii):  Since $g$ is not geometric at $\ell$, by definition there exists a root $\alpha \in \overline{K}$ of $X^\ell-g$ such that $K(\alpha)/K$ is either not a proper field extension or such that $K(\alpha) \cap \overline{\mathbb{F}}_q \neq \fq$. In the former case, we may write $g = \mu b^{\ell}$ where $b = \alpha$ and $\mu=1$, and we're done.
	
We therefore assume $K(\alpha)/K$ to be a proper field extension which is not geometric, i.e. $M := K(\alpha) \cap \overline{\mathbb{F}}_q \neq \fq$.  Note that since $g$ is not an $\ell^{th}$ power the polynomial $X^\ell-g$ is irreducible over $K$ ~\cite[VI \S9]{lang2012algebra}.  
Thus $X^\ell-g$ is the minimal polynomial of $\alpha$, and we conclude that $[K(\alpha) :K] = \ell$.

Since $M \supsetneq \fq$ we further note that $K(\alpha) \supseteq MK \supsetneq K$. Since $[K(\alpha):K] = \ell$ is prime, it follows that $MK = K(\alpha)$. Next, note that $M/\fq$ is a finite extension of finite fields and hence Galois.  It then further follows that the extension of composita $MK/\fq K = K(\alpha)/K$ is also Galois.  Thus $X^\ell-g$ splits into distinct linear factors in $K(\alpha)$, and we may write
\[X^\ell-g = \prod_{i=0}^{\ell-1}(x - \zeta_{\ell}^{i}\alpha),\]
where $\zeta_{\ell} \in K(\alpha)$ denotes a fixed primitive $\ell^{th}$ root of unity.

Note then that $K(\alpha) \supseteq K(\zeta_{\ell}) \supseteq K$, and moreover that $K(\alpha) \neq K(\zeta_{\ell})$ since $[K(\zeta_{\ell}):K]\leq \ell-1 < \ell$.  Therefore, since $\ell$ is prime, it follows that $K(\zeta_{\ell}) = K$.  Since elements in $K$ of finite order lie in $K \cap \overline{\mathbb{F}}_q = \fq$, it follows that $\zeta_{\ell} \in \fq^{\times}$.  Noting that 
$\zeta_{\ell} \in \fq^{\times}$ if and only if $\ell \mid q-1$, we further conclude that $\ell \mid q-1$. 
	
	By ~\cite[Proposition 8.1]{rosen2002number}, we find that $[M:\fq] = [K(\alpha):K] =\ell.$  Since 
$\ell \mid q-1$, we moreover find that
\[\# \fq^{\times,\ell} = \frac{q-1}{\ell}< q-1 = \# \fq^{\times},\]
and thus, in particular, that there exists an element $\mu \in \fq^{\times}$ that is \textit{not} an $\ell^{th}$ power.  Again by~\cite[VI \S9]{lang2012algebra}, we find that the polynomial $X^\ell-\mu$ is irreducible over $\fq$, and therefore 
$[\fq(\beta):\fq]= \ell$, where $\beta \in \overline{\mathbb{F}}_q$ is any root of the polynomial $X^{\ell}-\mu$.  By the uniqueness of finite field extensions, it further follows that $M = \fq(\beta)$. From~\cite[Proposition 8.1]{rosen2002number} it then also follows that $\{1,\beta, \dots, \beta^{\ell-1}\}$ form a basis for $K(\alpha)/K$, and therefore $K(\alpha) = K(\beta)$.
	Hence there exist $b_i \in K$, $i = 0,1, \hdots, \ell-1$ such that
	\[
	\alpha = \sum_{i=0}^{\ell-1}b_i \beta^i.
	\]
	Let $\sigma$ be a non-trivial element of $\mathrm{Gal}(K(\alpha)/K)$.  Then $\sigma(\alpha)$ is a root of $X^{\ell}-g$, 
and we may write $\sigma(\alpha) = \zeta_\ell^n \alpha$ for some $n \in \{1, \hdots, \ell-1\}$.  Similarly, $\sigma(\beta)$ is a root of $X^{\ell}-\mu$, i.e. $\sigma(\beta) = \zeta_\ell^m \beta$ for some $m \in \{1, \hdots, \ell-1\}$. Hence
	\[
\sum_{i=0}^{\ell-1}b_i\zeta_\ell^n  \beta^i=	\zeta_\ell^n \alpha = \sigma(\alpha) = \sigma \left( \sum_{i=0}^{\ell-1}b_i \beta^i \right) = \sum_{i=0}^{\ell-1}b_i \zeta_{\ell}^{mi} \beta^i.
	\]
	Since $\{1, \beta, \hdots, \beta^{\ell-1}\}$ are linearly independent over $K$, it follows that 
		\[
b_i\zeta_\ell^n =b_i \zeta_{\ell}^{mi} \quad \textnormal{ for all } \quad 0 \leq i \leq \ell - 1.
	\]
Whenever $n \not \equiv mi \pmod \ell$ this implies $b_{i} = 0$.  Note that there exists a unique $0 \leq i_0 \leq \ell-1$ such that $n \equiv mi_0 \pmod \ell$. It follows that $\alpha = b_{i_{0}}\beta^{i_{0}}$, and therefore
	\[
	g = \alpha^\ell = \tilde{\mu} b_{i_0}^\ell,
	\]
where $\tilde{\mu} = \mu^{i_{0}} \in \fq^{\times}$. This shows the desired claim.

	(ii) $\implies$ (iii): Set $\tilde{g} = \sqrt[\ell]{\mu} b$, where $\sqrt[\ell]{\mu}$ denotes any root of $X^\ell- \mu$ in $\overline{\mathbb{F}}_q$. \\

	(iii) $\implies$ (i): If $\tilde{g} \in K$ then $g$ is an $\ell^{th}$ power, and therefore $g$ is clearly not geometric at $\ell$. So, assume $\tilde{g} \notin K$. Then $K \subsetneq K(\tilde{g}) \subset K \overline{\mathbb{F}}_q$.  Since $K(\tilde{g})/K$ is a proper finite extension, in fact $K(\tilde{g}) \subseteq K\mathbb{F}_{q^{n}}$, for some $n > 1$.  Moreover, since $\textnormal{Gal}(K\mathbb{F}_{q^{n}}/K) \cong \textnormal{Gal}(\mathbb{F}_{q^{n}}/\fq) \cong \Z/n\Z$, there exists a unique subgroup of $\textnormal{Gal}(K\mathbb{F}_{q^{n}}/K)$ of any given order dividing $n$. By the fundamental theorem of Galois theory, we then find that there exists a unique subextension of $K\mathbb{F}_{q^{n}}/K$ of degree $\ell = [K(\tilde{g}):K]$, and thus may conclude that $ K(\tilde{g}) = K\mathbb{F}_{q^{\ell}} $.  By \cite[Proposition 8.3]{rosen2002number}, it follows that 
\[K(\tilde{g})\cap \overline{\mathbb{F}}_q = K\mathbb{F}_{q^{\ell}} \cap \overline{\mathbb{F}}_q = \mathbb{F}_{q^{\ell}}\supsetneq \fq,\]
i.e. $g$ is not geometric at $\ell$, as desired.
	\end{proof}
The first part of the proof of Lemma~\ref{le.equivalent_characterisations_non_geometric} shows that if $g$ is not geometric at a prime $\ell$ such that $\ell \nmid q-1$, then $g$ is already a full $\ell^{th}$ power in $K$.  In particular, suppose $g = \tilde{g}^{\ell}$, for some prime $\ell \nmid q-1$ such that $\ell \mid q^{n}-1$. Consider a closed point $\p \in X$, where $\deg \p = n$, and such that $g$ is regular at $\p$.   Then by a similar argument to that provided in Section \ref{Fq[T] Artin}, we find that \[\mathrm{ord}_{\p}(g)=\ell \cdot  \mathrm{ord}_{\p}(\tilde{g}) \leq q^{n}-1 \Longrightarrow \mathrm{ord}_{\p}(g) \leq \frac{q-1}{\ell} < q-1,\]
where we moreover note that since $g$ regular at $\p$ then $\tilde{g}$ is also regular at $\p$.  In other words, we conclude as follows:

\begin{cor}
Suppose $g \in K^{\times}$ is not geometric at some prime $\ell \nmid q-1$ such that $\ell \mid q^{n}-1$.  Then $g$ is not a primitive root modulo any closed point $\p \in X$ of degree $n$, i.e. $N_{X}(g,n) = 0.$
\end{cor}

\begin{lemma} \label{lem.only_finitely_many_non_geometric}
	Let $g \in K \setminus \Fq$, and let	$\mathscr{P}_g$ denote the set of primes $\ell \neq p$ at which $g$ is not geometric.  Then $\mathscr{P}_g$ is finite.
\end{lemma}
\begin{proof}
By  \cite[\href{https://stacks.math.columbia.edu/tag/035Q}{Lemma 035Q}]{stacks-project} and  \cite[\href{https://stacks.math.columbia.edu/tag/0GK4}{Lemma 0GK4}]{stacks-project}, there exists a geometrically integral normal projective variety $X_{\nu}$ over $\fq$, such that $\Fq(X_{\nu}) \cong {\mathbb{F}}_q(X)$.  In particular, since $X_\nu$ is a geometrically integral projective variety, by~\cite[Proposition 2.2.22]{poonen2017rational} we find that the global sections are given by $\mathcal{O}_{X_\nu}(X_\nu) = \Fq$.

Suppose $\mathscr{P}_g$ is infinite. Note that it suffices to show that $g$ lies in the global sections $\mathcal{O}_{X_\nu}(X_\nu)$.

By Lemma~\ref{le.equivalent_characterisations_non_geometric}  there exists an arbitrarily large $\ell \in \mathbb{N}$ such that $g = \mu b^\ell$, where $\mu \in \fq$ and $b \in K$. Note that as $Y$ ranges over prime divisors of $X_{\nu}$, the maximum value  of $\lvert v_Y(g) \rvert$ is bounded by $\textnormal{deg}(g)$. Choose $\ell > \textnormal{deg}(g)$ such that $g = \mu b^{\ell}$. Then for any prime divisor $Y \subset X_\nu$, 	we find that 
	\[
	v_Y(g) = v_Y(\mu) +  \ell \cdot v_Y(b) = \ell \cdot v_Y(b),
	\]
	Thus $v_Y(g) = 0$ for any prime divisor $Y \subset X_\nu$, and in particular $v_Y(g) \geq 0$ for all prime divisors $Y$. It follows from~\cite[Proposition 6.3A]{hartshorne2013algebraic} that $g \in \mathcal{O}_{X_\nu}(X_\nu)$, as desired.
\end{proof}

\section{A Bound on Exponential Sums}\label{Exponential_Sum}

One of the key ingredients of the proof of Theorem~\ref{thm.APC_general} is the following estimate for exponential sums, which is of independent interest. As we were unable to find a suitable result 
of our desired form in the existing literature, we present a proof here.

\begin{proposition} \label{prop:exponential_sum_estimate}
Let $X$ be a geometrically integral projective variety of dimension $r$.	Let $\chi \in \widehat{\mathbb{F}_{q}^{\times}}$ be a non-trivial character of order $\delta >1$. Let $g \in  \fq(X)$ and assume that there exists a prime $\ell \mid \delta$ such that $g$ is geometric at $\ell$, i.e. that $g$ is not of the form  $g = \tilde{g}^\ell$ for any $\tilde{g} \in \overline{\mathbb{F}}_q(X)$. Let $\mathcal{R}_g \subset X(\mathbb{F}_{q})$ denote the set of $\mathbb{F}_{q}$-rational points on $X$ that are neither zeroes nor poles of $g$. Then
	\begin{equation} \label{eq.exp_sum_estimate}
		\sum_{\rho \in \mathcal{R}_g} \chi(g(\rho)) \ll_X  q^{r-1/2}.
	\end{equation}
\end{proposition}

\begin{proof}
	Let $U \subset X$ be an affine open subset of $X$ on which $g$ is invertible, i.e. such that $g$ has neither poles nor zeroes on $U$. It suffices to show the estimate for the sum over $ U(\Fq)$.  Indeed, since $X \setminus U$ is a proper closed subset of $X$, then since $X$ is irreducible, dim$(X \setminus U)<r$. Therefore by the Lang--Weil bounds~\cite{lang_weil} the number of $\fq$-rational points of 
$X \setminus U$ is bounded by $O(q^{r-1})$. 
	
	By Noether's normalization lemma there exists a finite surjective morphism $U \rightarrow \mathbb{A}^r$, where $\mathbb{A}^r:=\textnormal{Spec}(\fq[x_{1},\dots, x_{r}])$. Obviously there also exists a surjective map $\mathbb{A}^{r} \rightarrow \mathbb{A}^{r-1}$ by projecting on the first $r-1$ coordinates, say. The composition of these maps yields a surjective morphism of locally finite type $\varphi \colon U \rightarrow \mathbb{A}^{r-1}$.
	
	From Chevalley's upper semicontinuity theorem (\cite[Théorème 13.1.3]{EGA4} and~\cite[Theorem 11.4.2]{vakil2017rising}) it follows that the elements $x \in U$ such that $\dim \varphi^{-1}(\varphi(x)) > 1$ holds lie in a proper closed subset of $U \subset X$, and thus has dimension at most $r-1$. Thus, by 
Lang--Weil, the number of rational points in this subset is bounded by $O(q^{r-1})$.  Note that if there is no $\fq$-rational point $\mathbb{A}^{r-1}(\fq)$ corresponding to $y \in \varphi(U)$, then $\varphi^{-1}(y)(\fq)$ is empty.  It therefore remains to estimate
	\[
	\sum_{\substack{y \in \varphi(U)(\fq) \\ \dim \varphi^{-1}(y) = 1}} \sum_{\rho \in \varphi^{-1}(y)(\Fq)} \chi(g(\rho)),
	\]
where, if $y \in \varphi(U)(\fq)$, then with a slight abuse of notation we write $\varphi^{-1}(y)$ to denote the fiber coming from the closed point in $\varphi(U)$ corresponding to $y$.
	On the fibres $\varphi^{-1}(y) \subset U$
where $\dim \varphi^{-1}(y) = 1$, we apply a theorem of Perelmuter~\cite[Theorem 2]{perel1969estimation}, which states the following. Let $\varphi^{-1}(y)_{\overline{\mathbb{F}}_q}$ refer to $\varphi^{-1}(y)$ after changing the base field to $\overline{\mathbb{F}}_q$, and suppose that $g$, when restricted to any irreducible component of $\varphi^{-1}(y)_{\overline{\mathbb{F}}_q}$, is not a $\delta^{th}$ power of any element in $\overline{\mathbb{F}}_q(X)$.  Then
	\[
	\sum_{\rho \in  \varphi^{-1}(y)(\Fq)} \chi(g(\rho)) \ll_X q^{1/2}
	\]
	uniformly in $y$.  Note that if $g$ is not an $\ell^{th}$ power for $\ell$ as in the assumption of the proposition, then $g$ is not a $\delta^{th}$ power. The remainder of the proof is thus concerned with showing that for generic $y \in \varphi(U)$, the element $g$ is not an $\ell^{th}$ power when restricted to an irreducible component of $\varphi^{-1}(y)_{\overline{\mathbb{F}}_q}$, and thus Perelmuter's theorem is applicable.

	 Call $y \in \mathbb{A}^{r-1}$ \emph{bad} if 
$y \in \varphi(U)$ and this occurs, and \emph{good} otherwise. We claim that there exists a constructible set $C \subset \mathbb{A}^{r-1}$ containing the generic point $\eta$, such that $C$ is contained in the set of good points.	
	by \cite[\href{https://stacks.math.columbia.edu/tag/005K}{Lemma 005K}]{stacks-project} we deduce that $C$ contains an open dense subset in $\mathbb{A}^{r-1}$ and so $\mathbb{A}^{r-1} \setminus C$ is contained in a proper closed subset of $\mathbb{A}^{r-1}$. Since $\mathbb{A}^{r-1} \setminus C$ contains the set of bad points, by Lang--Weil the number of bad $\fq$-rational points on $\mathbb{A}^{r-1}$ is bounded by $O(q^{r-2})$. Therefore, 	
	by trivially bounding the character sums for the fibers coming from $\mathbb{A}^{r-1} \setminus C$, we find that
	
\[
	\sum_{\substack{y \in (\mathbb{A}^{r-1} \setminus C)(\Fq) \\ \dim \varphi^{-1}(y) = 1}} \sum_{\rho \in \varphi^{-1}(y)(\Fq)} \chi(g(\rho))= O(q^{r-1}).
	\]		
	To show the claim made above we will employ~\cite[\href{https://stacks.math.columbia.edu/tag/055B}{Lemma 055B}]{stacks-project}. This states that if $h \colon Z \rightarrow Y$ is a morphism of finite presentation between schemes, and if $n_h \colon Y \rightarrow \{0,1, \hdots \}$ is the number of  irreducible components of the fibre $h^{-1}(y)$ after base change to $\overline{\mathbb{F}}_q$ then for a positive integer $n$ the set
	\[E_n = \{y \in Y \colon n_h(y) = n \}\] 
	is constructible.

	Recall that $U$ is an affine open subset of $X$ and is of the form $U = \mathrm{Spec} (R)$, where $R = \Fq[x_1, \hdots, x_n]/I$ for some ideal $I$. Since $g$ is invertible on $U$, we may consider $g$ restricted to $U$ as an element in $R$. 
	Consider $U_g = \mathrm{Spec}(R[z]/(z^\ell-g))$, and note that we have a natural map $ \psi \colon U_g \rightarrow U$ induced by the inclusion map $R \hookrightarrow R[z]/(z^\ell-g)$ and write $f$ for the composition $f = \varphi \circ \psi \colon U_g \rightarrow \mathbb{A}^{r-1}$.
	
Since $\psi$ is a surjective morphism, we see that $n_{f}(y) \geq n_{\varphi}(y)$ for all $y \in \mathbb{A}^{r-1}$.  Note furthermore that since all the schemes involved are noetherian, $f$ is locally of finite type.  It thus follows from~\cite[\href{https://stacks.math.columbia.edu/tag/01TX}{Lemma 01TX}]{stacks-project} that $f$ is of finite presentation.  By~\cite[\href{https://stacks.math.columbia.edu/tag/055B}{Lemma 055B}]{stacks-project}	
we then find that the set
\[
C = \{y \in \mathbb{A}^{r-1} \colon n_f(y) = 1 \} 
\]	
is constructible.

Note that the generic fibre $\varphi^{-1}(\eta)$ is integral with function field isomorphic to $K$ and in particular it is also integral after changing base to $\overline{\mathbb{F}}_q$. Further since $g$ was assumed not to be an $\ell^{th}$ power in $\overline{\mathbb{F}}_q(X)$, we find that $n_{f}(\eta) =1$ and therefore $\eta \in C$.
Further note that if $n_f(y) = n_\varphi(y)$ then $y$ is good. Otherwise, if $y$ is bad then by construction we have $n_f(y) > n_\varphi(y)$.
\end{proof}

\section{Proof of Theorem~\ref{thm.APC_general}}\label{Proof_of_Theorem}
Consider a finite cyclic group $G$ of order $M$, and let $\widehat{G} := \mathrm{Hom}(G,\C^{\times})$ denote its group of characters. Let 
\[
f_G(g) := \frac{\varphi(M)}{M} \prod_{\substack{p \mid M\\ p 
\textnormal{ prime}}} \left(1-\frac{1}{\varphi(p)}\sum_{\substack{\chi \in \widehat{G} \\ \mathrm{ord} \chi = p}} \chi(g) \right),
\]
We begin by noting the following general formula (see also \cite{JensenMurty2000, landau1927}):

\begin{lemma}\label{lem.characteristic_function}
For $g \in G$, we have that 
\begin{equation*}
f_G(g)=\frac{\varphi(M)}{M}\sum_{d \mid M} \frac{\mu(d)}{\varphi(d)} \sum_{\substack{\chi \in \widehat{G} \\ \mathrm{ord} \chi = d}} \chi(g)= \begin{cases}
	1 \quad &\text{if $g$ generates $G$} \\
	0 &\text{otherwise.}
	\end{cases}
\end{equation*}
\end{lemma}
\begin{proof}
Suppose $g \in G$ does not generate $G$.  Then we may write $g = h^p$ for some $h \in G$, where $p \mid M$ is prime.  By the fundamental theory of cyclic groups, $G$ has a unique subgroup of order $d$, for every $d|M$.  Such a group is moreover cyclic; and its generators are precisely the elements in $G$ of order $d$.  It follows that there are precisely $\varphi(d)$ characters of order $d$ in $\widehat{G}$, since $\widehat{G} \cong G$.  Thus
\[
\sum_{\substack{\chi \in \widehat{G} \\ \mathrm{ord} \chi = p}} \chi(g)=\sum_{\substack{\chi \in \widehat{G} \\ \mathrm{ord} \chi = p}} \chi^{p}(h) = \varphi(p),
\]
from which it follows that $f_G(g) = 0$. Alternatively, suppose $g \in G$ generates $G$. By orthogonality we then find that
\[
\sum_{\substack{\chi \in \widehat{G} \\ \mathrm{ord} \chi = p}} \chi(g) = -1,
\]
and moreover by M{\"o}bius inversion,
\[
\frac{M}{\varphi(M)} = \prod_{p \mid M} \left(1+\frac{1}{p-1} \right).
\]
We thus conclude, as desired, that
\[
f_G(g) = \begin{cases}
	1 \quad &\text{if $g$ generates $G$} \\
	0 &\text{otherwise.}
	\end{cases}
\]
Finally, by the Chinese remainder theorem, we note that for $(d_1,d_2)=1$ and $g \in G$,
\[
\sum_{\substack{\chi \in \widehat{G} \\ \mathrm{ord} \chi = d_1d_2}} \chi(g) = \Bigg( \sum_{\substack{\psi \in \widehat{G} \\ \mathrm{ord} \psi = d_1}} \psi(g)\Bigg) \Bigg( \sum_{\substack{\phi \in \widehat{G} \\ \mathrm{ord} \phi = d_2}} \phi(g)\Bigg)
\]
By multiplicativity we thus conclude that
\begin{align*} 
f_G(g) &= \frac{\varphi(M)}{M} \prod_{p \mid M} \left(1-\frac{1}{\varphi(p)}\sum_{\substack{\chi \in \widehat{G} \\ \mathrm{ord} \chi = p}} \chi(g)  \right)\\
&=\frac{\varphi(M)}{M}\sum_{d \mid M} \frac{\mu(d)}{\varphi(d)} \sum_{\substack{\chi \in \widehat{G} \\ \mathrm{ord} \chi = d}} \chi(g). \qedhere
\end{align*}
\end{proof}

Let $X/\Fq$ be a geometrically integral projective variety. Write $K$ for the function field of $X$ and fix algebraic closures $\overline{\mathbb{F}}_q$ and $\overline{K}$. Let $\mathfrak{p} \in  X$ be a closed point of degree $n$,  i.e. $[\kappa_{\mathfrak{p}} \colon \Fq] = n$.
 There then exists an isomorphism 
\[
\kappa_{\mathfrak{p}}^{\times} \cong \mathbb{F}_{q^n}^{\times} \subset \overline{\mathbb{F}}_{q},
\]
which may be explicitly described as follows.  Let $[f] \in \kappa_{\mathfrak{p}}^{\times}$ denote a residue class represented by some $f \in K$ that is regular at $\mathfrak{p}$, and let $\rho \in X(\mathbb{F}_{q^{n}})$ denote an $\mathbb{F}_{q^n}$-rational point of $X$ corresponding to $\mathfrak{p}$.  Since $f \in K$ is regular at $\mathfrak{p}$, $f$ does not have a pole at $\rho$, and in fact $f(\rho) \in \mathbb{F}_{q^n}^{\times}$, since $f$ does not vanish at $\mathfrak{p}$.  It follows that the map $[f] \mapsto f(\rho)$ is a well-defined isomorphism.\\
\\
For a fixed $g \in K$, we thus find that $g$ is a primitive root mod $\mathfrak{p}$ if and only if $g$ is regular at $\mathfrak{p}$ and its image $g(\rho)$ under this isomorphism generates $\mathbb{F}_{q^n}^{\times}$. Note that for any closed point of degree $n$ there exist $n$ different $\mathbb{F}_{q^n}$-rational points on $X$ corresponding to it obtained by the action of the Frobenius element, see~\cite[Lemma II.4.4]{hartshorne2013algebraic}. For an affine open $U \subset X$ this was explained in Section~\ref{geometric background}. It would also be sufficient to only consider these rational points on $U$, since the number of $\mathbb{F}_{q^n}$-rational points on $X \setminus U$ is $O(q^{r(n-1)})$ by Lang--Weil.\\
\\
Let $\mathcal{R}^{(n)}_g \subset X(\mathbb{F}_{q^{n}})$ denote the set of $\mathbb{F}_{q^{n}}$-rational points on $X$ that are neither zeroes nor poles of $g$, so that $g(\rho) \in \mathbb{F}_{q^{n}}^{\times}$ for any $\rho \in \mathcal{R}^{(n)}_g$.  If we then consider any $\rho \in \mathcal{R}_g^{(n)}$ corresponding to a closed point $\mathfrak{q}$ with $\deg (\mathfrak{q}) < n$ then $g(\rho)$ is contained in a proper subfield of $\mathbb{F}_{q^n}$ and therefore will not generate $\mathbb{F}_{q^n}^{\times}$. Here we may consider this subfield as a subset of $\mathbb{F}_{q^n}$ since we fixed an algebraic closure $\overline{\mathbb{F}}_q$. Further note that if $\rho \in X(\mathbb{F}_{q^{n}}) \setminus \mathcal{R}_g^{(n)}$, then $\rho$ corresponds to a closed point $\p \in X$ at which $g$ either vanishes or is not regular.  We thus find that
\begin{equation}\label{closed_points_count_rational_points}
N_X(g,n) = \frac{1}{n}	 \#\{ \rho \in \mathcal{R}_g^{(n)} \colon \langle g(\rho) \rangle = \mathbb{F}_{q^n}^{\times} \}.
\end{equation}
For positive integers  $k,n \in \N$, consider \textit{Ramanujan's sum}
\begin{equation}\label{ramanujan_sum}
c_k(m) := \sum_{\substack{1 \leq a \leq k \\(a,k) = 1}} e \left(\frac{am}{k} \right),
\end{equation}
and note that for a rational prime $\ell \in \mathbb{N}$,
\begin{equation} \label{eq.ramanujan_sums_simple_property}
c_\ell(m) = \begin{cases}
	-1 \quad &\text{if $\ell \nmid m$}, \\
	\varphi(\ell) &\text{if $\ell \mid m,$}.
	\end{cases}
\end{equation}
We prove the following lemma:

\begin{lemma}
Let $X/\Fq$ be a geometrically integral projective variety with function field $K=\fq(X)$.  Let $\mathscr{P}_g$ denote the set of primes $\ell \neq p$ at which $g$ is not geometric, and let $\mathcal{R}^{(n)}_g \subset X(\mathbb{F}_{q^{n}})$ denote the set of $\mathbb{F}_{q^{n}}$-rational points on $X$ that are neither zeroes nor poles of $g$.  Then 
\begin{equation}\label{eq:count_product_form}
N_X(g,n) = \rho_g(n) \frac{\varphi(q^n-1)}{n(q^n-1)}  \sum_{\rho \in \mathcal{R}_g^{(n)}} \prod_{\substack{\ell \mid q^n-1 \\ \ell \notin \mathscr{P}_g}} \left(1-\frac{1}{\varphi(\ell)}\sum_{\substack{\chi \in \widehat{G} \\ \mathrm{ord} \chi = \ell}} \chi(g(\rho))  \right),
\end{equation}
where
\begin{equation}\label{rho_constant}
\rho_g(n) := \prod_{\substack{\ell \mid q^n-1 \\ \ell \in \mathscr{P}_g}} \left(1-\frac{c_\ell(q^{n-1} + \cdots + 1)}{\varphi(\ell)} \right),
\end{equation}
and $c_k(m)$ denotes \textit{Ramanujan's sum}.
\end{lemma}
\begin{proof}
By \eqref{closed_points_count_rational_points}
and Lemma~\ref{lem.characteristic_function}, we find that 
\begin{align}\label{eq.char_fun_after_iso}
N_X(g,n) = \frac{\varphi(q^n-1)}{n(q^n-1)}  \sum_{\rho \in \mathcal{R}_g^{(n)}} \prod_{\substack{\ell \mid q^n-1 \\ \ell \textnormal{ prime}}} \left(1-\frac{1}{\varphi(\ell)}\sum_{\substack{\chi \in \widehat{G} \\ \mathrm{ord} \chi = \ell}} \chi(g(\rho))  \right).
\end{align}
Note that if $g$ is an $\ell^{th}$ power for some prime $\ell \mid q-1$, then $N_X(g,n)=0$ for all $n \geq 1$.  This can be seen upon noting that there are precisely $\varphi(\ell)$ characters $\chi \in \widehat{G}$ of order $\ell$, and that $\ell \mid q^{n}-1$ for any $n \geq 1$, when $\ell \mid q-1$.  Alternatively, this follows from elementary group theoretic considerations, as noted in Section \ref{Fq[T] Artin}.  Since  \eqref{eq:count_product_form} holds in such a case, we may henceforth assume that $g \in K$ is not an $\ell^{th}$ power for any prime $\ell \mid q-1$.\\
\\
Let $\ell \mid q^n-1$ be a prime such that $g$ is not geometric, that is, $\ell \in \mathscr{P}_g$. By Lemma~\ref{le.equivalent_characterisations_non_geometric}, we may then write $g = \mu_\ell b_\ell^\ell$ for some $\mu_\ell \in \mathbb{F}_q^\times$ and some $b_\ell \in K^{\times}$. Therefore, if $\mathrm{ord}(\chi) = \ell$ we have $\chi(g(\rho)) = \chi(\mu_\ell)$ for any $\rho \in \mathcal{R}_g^{(n)}$.  Let $r$ denote the order of $\mu_\ell \in \fq^\times$, and write $\mu_{\ell} = \xi^{a}$ for some generator $\xi \in \mathbb{F}_q^\times$.  Then $\xi^{ar}=1$, from which it follows that $\frac{q-1}{r} \mid a$.  Suppose, now, that $\ell \mid \frac{q-1}{r}$.  Then $\ell \mid a$, so we may write $\mu_{\ell} = (\xi^{\frac{a}{\ell}})^{\ell}$.  In other words, $\mu_{\ell}$ is an $\ell^{th}$ power where, moreover, $\ell \mid q-1$.  From this it further follows that $g$ is an $\ell^{th}$ power in $K$, contradicting our assumption.  We thus conclude that $\ell \nmid  \frac{q-1}{r}$ or, equivalently, that $(\ell, \frac{q-1}{r}) = 1$.\\
\\
For $x \in \mathbb{R}$, write $e(x) := e^{2 \pi i x}$. 
Consider an embedding
\[
\psi \colon \mathbb{F}_{q^n}^{\times} \hookrightarrow \mathbb{C}^{\times}
\]
such that for $\mu_\ell \in \mathbb{F}_{q}^{\times} \subset \mathbb{F}_{q^n}^{\times}$ as above, we have
\[
\psi(\mu_\ell) = e\left( \frac{1}{r} \right).
\]
A character $\chi \colon \mathbb{F}_{q^n}^{\times} \rightarrow \mathbb{C}^{\times}$ of order $\ell$ then acts on an element $\alpha \in \mathbb{F}_{q^n}^{\times}$ via
\[
\chi(\alpha) = \psi(\alpha)^{\frac{(q^n-1)a}{\ell}},
\]
for some $a \in \{1, \hdots, \ell-1 \}$.
It follows that for any given $\rho \in \mathcal{R}_g^{(n)}$,
\begin{align*}
 \sum_{\mathrm{ord} \chi = \ell} \chi(g(\rho)) &= \sum_{1 \leq a \leq \ell-1} e \left(\frac{1}{r} \right)^{\frac{(q^n-1)a}{\ell}}
= \sum_{ 1 \leq a \leq \ell -1} e\left( \frac{a(q^n-1)}{\ell r} \right)=c_\ell\left(\frac{q^n-1	}{r} \right),
\end{align*}
where $c_{\ell}\left(\frac{q^n-1}{r}\right)$ denotes a Ramanujan sum.  Since $(\ell, \frac{q-1}{r}) = 1$ we have $\ell \mid \frac{q^n-1}{r}$ if and only if $\ell \mid q^{n-1} + \cdots + 1$, i.e. that 
\[
c_\ell\left(\frac{q^n-1	}{r} \right) = c_\ell\left(q^{n-1} + \cdots + 1 \right)= \begin{cases}
	-1 \quad &\text{if $\ell \nmid \frac{q^n-1}{r}$} \\
	\varphi(\ell) &\text{otherwise.}
\end{cases}
\]
Thus, for any $\rho \in \mathcal{R}_g^{(n)}$,
\[ \prod_{\substack{\ell \mid q^n-1\\ \ell \in \mathscr{P}_g}} \left(1-\frac{1}{\varphi(\ell)}\sum_{\substack{\chi \in \widehat{G} \\ \mathrm{ord} \chi = \ell}} \chi(g(\rho))  \right) = \prod_{\substack{\ell \mid q^n-1 \\ \ell \in \mathscr{P}_g}} \left(1-\frac{c_\ell(q^{n-1} + \cdots + 1)}{\varphi(\ell)} \right) = \rho_g(n),\]
and together with \eqref{eq.char_fun_after_iso}, this yields \eqref{eq:count_product_form}.
\end{proof}
\begin{proof}[Proof of Theorem \ref{thm.APC_general}]
For an integer $\delta \in \mathbb{N}$ write $(\delta,\mathscr{P}_g) = 1$ if 
$(\delta,\ell) = 1$ for every prime $\ell \in \mathscr{P}_g$. As in the proof of Lemma \ref{lem.characteristic_function} 
we may then expand \eqref{eq:count_product_form} to obtain
\[
N_X(g,n) = \rho_g(n) \frac{\varphi(q^n-1)}{n(q^n-1)}   \sum_{\substack{ \delta \mid q^n-1 \\ (\delta, \mathscr{P}_g) = 1}} \frac{\mu(\delta)}{\varphi(\delta)} \sum_{\mathrm{ord} \chi = \delta} \sum_{\rho \in \mathcal{R}_g^{(n)}} \chi(g(\rho)).
\]
By the Lang--Weil bounds~\cite{lang_weil}, 
the number of $\mathbb{F}_{q^{n}}-$rational points on $X$, denoted $\#X(\mathbb{F}_{q^{n}})$, is given by
\[\lvert \#X(\mathbb{F}_{q^{n}}) - q^{nr}\rvert \ll_X q^{n(r-1/2)}.\]
Noting, moreover, that $g$ has at most $m$ poles and zeroes, it follows that for fixed $g$,
\begin{align*} \label{eq.lang_weil}
	\lvert \# \mathcal{R}_g^{(n)}-q^{nr}\rvert &\leq m+O_{X}\left(q^{n(r-1/2)}\right) = O_{g,X}\left(q^{n(r-1/2)}\right),
\end{align*}
and thus the contribution from the trivial character $\chi_{0}$ is given by
\[
\sum_{\rho \in \mathcal{R}_g^{(n)}} \chi_{0}(g(\rho)) = \# \mathcal{R}_g^{(n)} = q^{nr} + O\left(q^{n(r-1/2)}\right).
\]
If $\delta \mid q^n-1$ is such that $(\delta, \mathscr{P}_g)=1$ and $\delta >1$, then by Proposition~\ref{prop:exponential_sum_estimate} we moreover find that for $\chi$ of order $\delta$,
\[
\sum_{\rho \in \mathcal{R}_g^{(n)}} \chi(g(\rho)) = O\left(q^{n(r-1/2)}\right).
\]

Combining the above observations, and applying the divisor bound $\sum_{\delta \mid q^n-1} \lvert \mu(\delta) \rvert = O_\varepsilon(q^{n  \varepsilon})$, we obtain
\begin{equation*}
	N_X(g,n) = \rho_g(n) \left( \frac{\varphi(q^n-1)q^{n(r-1)}}{n}+ O_\varepsilon\left(q^{n(r-1/2+\varepsilon)} \right)\right),
\end{equation*} 
as desired.
\end{proof}
\section{Proof of Theorem \ref{Artin Conjecture}}\label{qualitative_results}
We deduce Theorem \ref{Artin Conjecture} from Theorem~\ref{thm.APC_general}, by demonstrating that $N_{X}(g,n) > 0$ for infinitely many $n \in \mathbb{N}$.  We begin by providing a more careful study of the correction factor $\rho_g(n)$:

\begin{lemma}\label{rho_description}
$\rho_g(n) > 0$ if and only if for all primes $\ell \in \mathscr{P}_g$ such that $\ell \mid q^n-1$ we have $\ell \mid q-1$ and $\ell \nmid n$.
\end{lemma}

\begin{proof}
Let $\ell$ be a prime dividing $q^n-1$ such that $g$ is not geometric at $\ell$. We proceed in two cases.  First, suppose $\ell \mid q-1$.  Then $q \equiv 1 \mod \ell$, and therefore
\[
q^{n-1} + \cdots + 1  \equiv 0 \mod \ell \iff n \equiv 0 \mod \ell.
\]
By~\eqref{eq.ramanujan_sums_simple_property}, we then find that $c_\ell(q^{n-1} + \cdots +1) = -1$ if and only if $\ell \nmid n$.  Otherwise, $c_\ell(q^{n-1} + \cdots +1) = \varphi(\ell)$, in which case $\rho_{n}(g)=0$.

Next, suppose $\ell \nmid q-1$. Since $ \ell \mid q^n-1$ by assumption, we then find that $\ell \mid q^{n-1} + \cdots + 1$.  By~\eqref{eq.ramanujan_sums_simple_property} it then follows that $c_\ell(q^{n-1} + \cdots + 1) = \varphi(\ell)$, and therefore $\rho_n(g) = 0$.  In conclusion, we find that $\rho_g(n) > 0$ if and only if  for all primes $\ell \in \mathscr{P}_g$ such that $\ell \mid q^n-1$ we have $\ell \mid q-1$ and $\ell \nmid n$, as desired.
\end{proof}

\begin{proof}[Proof of Theorem \ref{Artin Conjecture}]
First note that if $K$ is any function field over $\fq$ of finite transcendence degree, then $K = \fq(X)$ is the function field of a geometrically integral, projective variety $X/\fq$.

Let $g \in K \setminus \fq$ and assume $g$ is not a full $\ell^{th}$ power for any prime $\ell \mid q-1$, so that ~\eqref{eq.number_roots_degree_n} holds. We wish to show that there exist infinitely many closed points $\mathfrak{p}$ of $X$ such that $g$ is a primitive root modulo $\mathfrak{p}$, i.e. that there exist infinitely many $n \in \mathbb{N}$ such that $N_{X}(g,n) \neq 0$.

Note that $\rho_g(n) \geq 1$ whenever $\rho_g(n) \neq 0$.  To show that $N_{X}(g,n) \neq 0$ infinitely often, it therefore suffices to show that there exist infinitely many $n \in \N$ such that $\rho_g(n) >0$. 
Let $\mathscr{P}_g = \{\ell_s: s \in S\}$ denote the set of primes $\ell \neq p$ at which $g$ is not geometric.  Since $g \not\in \fq$ we note that $\mathscr{P}_g$ is a finite set, by Lemma~\ref{lem.only_finitely_many_non_geometric}. Let $I \subset S$ be such that $i \in I$ whenever $\ell_i \mid q - 1$, and let $J = S \setminus I$ be such that $j \in J$ whenever $\ell_j \nmid q-1$. Given $m \in \N$ we then consider
	\begin{equation} \label{eq.natural_numbers}
			n = 1 + m \prod_{i \in I}\ell_i \prod_{j \in J}(\ell_j-1).
	\end{equation}
	We claim that the set of primes in $\mathscr{P}_g$, which divide $q^n-1$, is precisely given by $\{\ell_i \colon i \in I\}$. Note first that $\ell_i \mid q^n-1$ for all $i \in I$ since, in fact for any $n \in \mathbb{N}$, we have that $(q-1) \mid q^n-1$. On the other hand, let $j_0 \in J$. Then $q \not\equiv 1 \mod \ell_{j_0}$ and thus
	\[
	q^n \equiv q^{1 + m \prod_{i \in I}\ell_i \prod_{j \in J}(\ell_j-1)} \equiv q \not\equiv 1 \mod \ell_{j_0}
	\]
	since $q^{\ell_{j_0}-1} \equiv 1 \mod \ell_{j_0}$ by Fermat's little theorem. Hence $\ell_{j_0} \nmid q^n-1$. 
	
	 Finally note that $n \not\equiv 0 \mod \ell_i$ for all $i \in I$. From Lemma
\ref{rho_description} we then find that $\rho_g(n) >0$.  The result now follows upon noting that there are infinitely many $n \in \mathbb{N}$ of the form in ~\eqref{eq.natural_numbers}. 
\end{proof}

\section{A Heuristic Interpretation}\label{Heuristic_Interpretation}
Artin arrived at the quantitative version of his primitive root conjecture using a well-known heuristic concerning the splitting properties of primes across the fields $\Q(\zeta_{\ell},g^{1/\ell})$, for varying primes $\ell \in \mathbb{N}$.  The classical correction factor $c_{g} \in \Q$ emerges upon taking into account relevant \textit{entanglements} between number fields of this form.  In this section we suggest an analogous heuristic, which interprets the correction factor $\rho_{g}(n) \in \Q$ in terms of splitting properties of primes in $K$.  In contrast to the classical setting, we manage to obtain the appropriate correction factor while still preserving the assumption that the various splitting conditions are independent, across the primes $\ell \in \mathbb{N}$.

In what follows, we restrict ourselves to the case in which $K=\fq(X)$ is a global function field, i.e. in which $X$ is a normal geometrically integral projective curve over $\fq$.  A \textit{prime} $P$ of $K$ is, by definition, a discrete valuation ring $O_{P}$ with maximal ideal $P$, such that $\fq \subset O_{P}$ and $\textnormal{Frac}(O_{P})= K$.
Since $K$ is a global function field, the stalk $\mathcal{O}_{X,\p}$ at each closed point $\mathfrak{p} \in X$ is a discrete valuation ring, and such stalks are in fact in $1:1$ correspondence with the primes in $K$ \cite[pg. 130]{hartshorne2013algebraic}.  If $L/K$ is a field extension then a prime $\mathfrak{P}$ of $L$ is said to \textit{lie above} $P$ (denoted $\mathfrak{P}|P$) if $O_{\mathfrak{P}} \cap K = O_P$.  In such a case, $O_{\mathfrak{P}}/\mathfrak{P}$ forms a vector space over $O_{P}/P$ of finite degree, and we refer to $f_{\mathfrak{P}/P}:=[O_{\mathfrak{P}}/\mathfrak{P}:O_{P}/P]$ as the \textit{residual degree} of $\mathfrak{P}$ over $P$.  We moreover find that $P O_{\mathfrak{P}} = \mathfrak{P}^{e_{\mathfrak{P}/P}}$, for some $e_{\mathfrak{P}/P} \in \Z_{>0}$, which we refer to as the \textit{ramification index} of $\mathfrak{P}$ over $P$.  We say that $P$ \emph{splits completely} in $L$ if $g = [L:K]$, i.e. if $f_{\mathfrak{P}_{i}/P}= e_{\mathfrak{P}_{i}/P}=1$ for all primes $\mathfrak{P}_{i}$ in $L$ lying above $P$.
Let $\{\mathfrak{P}_{1},\dots,\mathfrak{P}_{h(P)}\}$ denote the complete set of primes in $L$ lying above $P$.  By ~\cite[Proposition 7.2]{rosen2002number}, we then find that
\[\sum_{i=1}^{h(P)}f_{\mathfrak{P}_{i}/P}\cdot e_{\mathfrak{P}_{i}/P} = [L:K].\]

Let $g \in K \setminus \fq$.  Let $\mathfrak{p} \in X$ be a closed point such that $g$ is regular at $\mathfrak{p}$. 
Such $g \in K$ then fails to be primitive modulo $\mathfrak{p}$ if and only if the prime $P$ corresponding to $\mathfrak{p}$
splits completely in $K_\ell := K(\sqrt[\ell]{g}, \zeta_\ell)$, the splitting field of $X^\ell-g$, for some prime $\ell \in \mathbb{N}$, where $\ell \nmid q$
 \cite[Lemma 10.1 and Proposition 10.6]{rosen2002number}. We may therefore formulate a heuristic for the density of primes $P$ of degree $n$ for which $g$ is a primitive root by understanding the density of primes $P$ in $K$ which split completely in $K_\ell$, for each prime $\ell \in \mathbb{N}$.

Suppose $g$ is not a full $\ell^{th}$ power, for any prime $\ell \mid q-1$.  Otherwise, $N_{X}(g,n)=0$, trivially.  Given a prime $\ell \in \mathbb{N}$, let
 \[\mathbb{P}_\ell := \mathbb{P}(P \text{ splits completely in } K_\ell \mid \deg(P) = n ).\]
Under the heuristic assumption that the splitting conditions of $P$ in $K_{\ell}$ are independent for the various fields $K_\ell$, we expect the desired density to be given by
\[
A = \prod_\ell \left(1- \mathbb{P}_\ell \right).
\]
Note that $\ell \mid q^n-1$ if and only if $P$ splits completely in $K(\zeta_\ell)$ (cf.~\cite[Proposition 10.2]{rosen2002number}).  Thus $\ell \nmid q^n-1$ implies that $P$ does not split completely in $K_{\ell}$, i.e. $\mathbb{P}_\ell = 0$ for all $\ell \nmid q^n-1$, and thus
\[
A = \prod_{\ell \mid q^n-1} \left(1- \mathbb{P}_\ell \right).
\]
Note that when $\ell \mid q^n-1$ then again by~\cite[Proposition 10.2]{rosen2002number} we find that
\[
\mathbb{P}_\ell = \mathbb{P}(P \text{ splits completely in } K_\ell \mid \deg(P) = n, P \text{ splits completely in } K(\zeta_\ell)).
\]
Now suppose $P$ splits completely in $K(\zeta_\ell)$, i.e. that $f_{\p/P} = e_{\p/P}= 1$ for every prime $\mathfrak{p}$ in $K(\zeta_{\ell})$ lying above $P$.  Let $\mathfrak{P}$ denote a prime in $K_{\ell}$ lying above some such $\mathfrak{p}|P$.  Since the residual degree and ramification index behave transitively in towers of field extensions, we find that
\[f_{\mathfrak{P}/P} = f_{\mathfrak{P}/\p}\cdot f_{\p/P} = f_{\mathfrak{P}/\p}\]
and similarly that
\[ 
e_{\mathfrak{P}/P} = e_{\mathfrak{P}/\p}\cdot e_{\p/P} = e_{\mathfrak{P}/\p},\]
for any such $\mathfrak{P}|\p$. In particular, we find that $f_{\mathfrak{P}/P}=e_{\mathfrak{P}/P}=1$ for every prime $\mathfrak{P}|P$ in $K_{\ell}$ if and only if $f_{\mathfrak{P}/\p} = e_{\mathfrak{P}/\p} = 1$ for every prime $\mathfrak{P}$ in $K_{\ell}$ lying above some such $\p|P$.  It follows that $P$ splits completely in $K_{\ell}$ if and only if every prime $\mathfrak{p}|P$ in $K(\zeta_{\ell})$ splits completely in $K_{\ell}$.\\
\\
Suppose $P$ is a prime in $K$ of degree $n$, such that $P$ splits completely in $K(\zeta_{\ell})$.  Note that since $f_{P/\p}=1$, we find that 
$[O_{\p}/\p:O_{P}/P]=1$ and therefore that
\[[O_{\p}/\p:\fq(\zeta_{\ell})]\cdot [\fq(\zeta_{\ell}):\fq] = [O_{\p}/\p:\fq]= [O_{P}/P:\fq].\]
It follows that
\[[O_{\p}/\p:\fq(\zeta_{\ell})]=\frac{[O_{P}/P:\fq]}{[\fq(\zeta_{\ell}):\fq]}=\frac{n}{\phi_{q}(\ell)},\]
where $\phi_{q}(\ell) := [\fq(\zeta_{\ell}):\fq]$ is equal to the multiplicative order of $q \mod \ell$ ~\cite[Proposition 10.2]{rosen2002number}.  In other words, we find that $P$ splits completely into $\phi_{q}(\ell)$ primes  $\p \subset K(\zeta_{\ell})$, such that $\deg \p:=[O_{\p}/\p:\fq(\zeta_{\ell})] = n/\phi_{q}(\ell)$
for each such $\p$.  By the \textit{prime polynomial theorem} ~\cite[Theorem 5.12]{rosen2002number}, we thus find that

\[\sum_{\substack{P \subset K \\ \deg P = n}} \# \left\{\p \subseteq K(\zeta_{\ell})
: \p|P  \right\} = \phi_{q}(\ell)\frac{q^{n}}{n}+O\left(q^{\frac{n}{2}}\right).\]

Note, moreover, that the constant field of $K(\zeta_{\ell})$ is equal to $\mathbb{F}_{q^{\phi_{q}(\ell)}}$.  Again from the prime polynomial theorem it then follows that

\[\# \left\{\p \subseteq K(\zeta_{\ell}): \deg \p = \frac{n}{\phi_{q}(\ell)} \right\} = \frac{q^{\phi_{q}(\ell)\frac{n}{\phi_{q}(\ell)}}}{n/\phi_{q}(\ell)}+O\left(q^{\frac{n}{2}}\right) = \phi_{q}(\ell)\frac{q^{n}}{n}+O\left(q^{\frac{n}{2}}\right).\]

We may thus conclude that

\[
\mathbb{P}_\ell = \mathbb{P}\left(\p \subseteq K(\zeta_{\ell})\text{ splits completely in } K_\ell \big \vert \deg \p = \frac{n}{\phi_{q}(\ell)}\right).
\]
Note that if $g$ is geometric at $\ell$ then $K_\ell/K(\zeta_\ell)$ is a geometric extension.  In such a case, we may apply \textit{Chebotarev's density theorem}~\cite[Theorem 9.13B]{rosen2002number}  to establish that the desired density is given by
\[
\mathbb{P}_\ell 
= \frac{1	}{[K_\ell:K(\zeta_\ell)]} = \frac{1}{\ell}.
\]
If $g$ is not geometric at $\ell$, then we may no longer apply Chebatarev's density theorem.  In such a case, however, we have sufficient information to compute $\mathbb{P}_\ell$ precisely. By Lemma~\ref{le.equivalent_characterisations_non_geometric}, we write $g = \mu b^\ell$ where $\mu \in \fq^\times$.  Let $r$ denote the order of $\mu$ in $\fq^\times$.  We may then find a generator $\zeta$ of $\fq^\times$ such that $\mu = \zeta^{\frac{q-1}{r}}$.  Indeed, suppose $\mu = \zeta_{0}^{\frac{q-1}{r}b_{0}}$ for a given generator $\zeta_{0}$, and where $(b_{0},r)=1$.  By the Chinese remainder theorem, we may find a   
$b \in \{b_{0}+kr: k \in \Z\}$ such that 
$(b,q-1)=1$.  We then write $\mu = \zeta^{\frac{q-1}{r}}$, where $\zeta = \zeta_{0}^{b}$ generates $\fq^{\times}$.

  By~\cite[Proposition 10.6]{rosen2002number} we find that a prime $P$ which splits completely in $K(\zeta_\ell)$ also splits completely in $K_\ell$, if and only if
\[
g^{\frac{q^n-1}{\ell}} \equiv \zeta^{\frac{q-1}{r} \cdot \frac{q^n-1}{\ell}} b^{\frac{q^n-1}{\ell} \cdot \ell} \equiv \zeta^{\frac{q-1}{r} \cdot \frac{q^n-1}{\ell}}\equiv 1 \mod P.
\]
Note that this in turn occurs if and only if $q-1 \mid \frac{q-1}{r} \cdot \frac{q^n-1}{\ell}$, enabling us to conclude that 
$$
\mathbb{P}_{\ell} = \left\{\begin{array}{ll}
1 & \text { if } \ell \mid \frac{q-1}{r} (q^{n-1} + \cdots + 1)\\
0 & \text { otherwise}.
\end{array}\right.
$$
In particular, since $\ell \mid q^{n}-1 = (q-1)(q^{n-1} + \cdots + 1),$ we find that 
$\mathbb{P}_{\ell} =1$ whenever $\ell \nmid q-1$.

So suppose $\ell \mid q-1$.  If $\ell \mid \frac{(q-1)}{r}$, then $\mu = \zeta^{\frac{q-1}{r}}$ is an $\ell^{th}$ power, in which case $g$ is also an $\ell^{th}$ power, contradicting our initial assumption.  We may therefore assume that $\ell \nmid \frac{(q-1)}{r}$.  In this case,
$\mathbb{P}_\ell = 0$ if and only if $\ell \nmid (q^{n-1} + \cdots + 1)$.  Since $q \equiv 1 \mod \ell$, we find that

\[
q^{n-1} + \cdots + 1 \equiv n \mod \ell,
\]
so that 

$$
\mathbb{P}_{\ell} = \left\{\begin{array}{ll}
1 & \text { if } n \equiv 0 \mod \ell\\
0 & \text { if } n \not \equiv 0 \mod \ell.
\end{array}\right.
$$
We thus conclude as follows.  Suppose $g$ is not a full $\ell^{th}$ power for any prime $\ell \mid q-1$.  If, for all primes $\ell \in \mathscr{P}_g$ such that $\ell \mid q^n-1$, we have $\ell \mid q-1$ and $n \not\equiv 0 \mod \ell$, then the density is given by
\[
A = \prod_{\substack{\ell \mid q^n-1 \\ \ell \notin \mathscr{P}_g}} \left( 1- \frac{1}{\ell}\right) \prod_{\substack{\ell \mid q^n-1 \\ \ell \in \mathscr{P}_g}} 1.
\]
Otherwise $A = 0$. In either case, $A$ is then given by
\begin{align*}
A &= \prod_{\substack{\ell \mid q^n-1}} \left( 1- \frac{1}{\ell} \right) \prod_{\substack{\ell \mid q^n-1 \\ \ell \in \mathscr{P}_g}} \left( 1- \frac{c_\ell\left(q^{n-1} + q^{n-2} + \cdots + 1\right)}{\varphi(\ell)} \right) = \frac{\varphi(q^n-1)}{q^n-1} \rho_g(n),
\end{align*}
as expected.

\bibliography{artin_arxiv2.bib}
\bibliographystyle{plain}

\end{document}